\documentclass[12pt,a4paper,oneside]{amsart}

\pdfoutput=1

\usepackage{color,latexsym,textcomp,lastpage,fancyhdr,calc,graphicx,asymptote,txfonts}
\usepackage{amsmath,amstext,amsthm,amssymb,amsxtra,amsfonts}
\usepackage[left=2.5cm,right=2.5cm,bottom=2cm,top=2cm]{geometry} 

\newtheorem{theorem}{Theorem}
\newtheorem{corollary}{Corollary}
\newtheorem{lemma}{Lemma}

\newtheorem{definition}{Definition}
\theoremstyle{definition}

\newcommand{\beql}[1]{\begin{equation}\label{#1}}
\newcommand{\eeq}{\end{equation}}
\newcommand{\comment}[1]{}

\newcommand{\Abs}[1]{{\left|{#1}\right|}}

\newcommand{\Floor}[1]{{\left\lfloor{#1}\right\rfloor}}

\newcommand{\Set}[1]{{\left\{{#1}\right\}}}

\newcommand{\RR}{{\mathbb R}}

\newcommand{\ZZ}{{\mathbb Z}}

\newcommand{\inner}[2]{{\langle #1, #2 \rangle}}

\newcommand{\ft}[1]{\widehat{#1}}

\newcounter{rem}
\newcounter{step}
\newcounter{mysec}
\newcounter{mysubsec}[mysec]
\newcounter{othm}
\def\theothm{\Alph{othm}}

\newcounter{fcap}
\newcommand{\mycaption}[2]{
 \refstepcounter{fcap}\label{#2}
 Figure \ref{#2}: {#1}
}

\begin{document}

\title{Size of orthogonal sets of exponentials for the disk}

\author[A. Iosevich]{{Alex Iosevich}}
\address{A.I.: Department of Mathematics, 915 Hylan Building, University of Rochester,
Rochester, NY 14627, U.S.A.}
\email{iosevich@math.rochester.edu}
\thanks{A.I.: Supported by NSF grant DMS10-45404}

\author[M. Kolountzakis]{{Mihail N. Kolountzakis}}
\address{M.K.: Department of Mathematics, University of Crete, Knossos Ave., GR-714 09, Iraklio, Greece}
\email{kolount@math.uoc.gr}
\thanks{M.K.: Supported by research grant No 3223 from the Univ.\ of Crete and by
NSF grant DMS10-45404 and grants of the University of Rochester, whose hospitality is gratefully
acknowledged}

\date{\today}

\begin{abstract}
Suppose $\Lambda \subseteq \RR^2$ has the property that any two exponentials with frequency from $\Lambda$
are orthogonal in the space $L^2(D)$, where $D \subseteq \RR^2$ is the unit disk.
Such sets $\Lambda$ are known to be finite but it is not known if their size is uniformly bounded.
We show that if there are two elements of $\Lambda$ which are distance $t$ apart then the size
of $\Lambda$ is $O(t)$. As a consequence we
improve a result of Iosevich and Jaming and show that $\Lambda$ has at most
$O(R^{2/3})$ elements in any disk of radius $R$.
\end{abstract}

\maketitle


\ 

\noindent{\bf AMS Primary Classification:} 42B99

\section{Introduction}

\subsection{Orthogonal sets of exponentials for domains in Euclidean space}
Let $\Omega \subseteq \RR^d$ be a bounded measurable set
and let us assume for simplicity that $\Omega$ has Lebesgue measure 1.
The concept of a spectrum of $\Omega$ that we deal with in this paper
was introduced by Fuglede \cite{fuglede1974operators} 
who was studying a problem of Segal on the extendability of the partial differential operators
(on $C_c(\Omega)$)
$$
\frac{\partial}{\partial x_1}, \frac{\partial}{\partial x_2}, \ldots, \frac{\partial}{\partial x_d}
$$
to commuting operators on all of $L^2(\Omega)$.

\begin{definition}
A set $\Lambda \subseteq \RR^d$
is called a {\em spectrum} of $\Omega$ (and $\Omega$ is said to be a {\em spectral set})
if the set of exponentials
$$
E(\Lambda) = \Set{e_\lambda(x)=e^{2\pi i \lambda\cdot x}:\ \lambda\in\Lambda}
$$
is a complete orthogonal set in $L^2(\Omega)$.
\end{definition}
(The inner product in $L^2(\Omega)$ is $\inner{f}{g} = \int_\Omega f \overline{g}$.)

It is easy to see (see, for instance, \cite{kolountzakis2004milano}) that the orthogonality of $E(\Lambda)$
is equivalent to the {\em packing condition}
\beql{packing-condition}
\sum_{\lambda\in\Lambda}\Abs{\ft{\chi_\Omega}}^2(x-\lambda) \le \Abs{\Omega}^2,\ \ \mbox{a.e. ($x$)},
\eeq
as well as to the condition
\beql{zeros-condition}
\Lambda-\Lambda \subseteq \Set{0} \cup \Set{\ft{\chi_\Omega}=0}.
\eeq
Here $\chi_\Omega$ is the indicator function of $\Omega$.

The orthogonality and completeness of $E(\Lambda)$ is in turn equivalent to the {\em tiling condition}
\beql{tiling-condition}
\sum_{\lambda\in\Lambda}\Abs{\ft{\chi_\Omega}}^2(x-\lambda) = \Abs{\Omega}^2,\ \ \mbox{a.e. ($x$)}.
\eeq
These equivalent conditions follow from the identity
$\inner{e_\lambda}{e_\mu} = \int_\Omega e_\lambda \overline{e_\mu} = \ft{\chi_\Omega}(\mu-\lambda)$
and from the completeness of all the exponentials in $L^2(\Omega)$.
Condition \eqref{packing-condition} is roughly expressing the validity of Bessel's inequality for the
system of exponentials $E(\Lambda)$ while condition \eqref{tiling-condition} says that Bessel's inequality
holds as equality.

If $\Lambda$ is a spectrum of $\Omega$ then so is any translate of $\Lambda$ but there may be other spectra as well.

{\em Example:} If $Q_d = (-1/2, 1/2)^d$ is the cube of
unit volume in $\RR^d$ then
$\ZZ^d$ is a spectrum of $Q_d$.
Let us remark here that
there are spectra of $Q_d$ which are very different from affine images of the lattice $\ZZ^d$
\cite{iosevich1998spectral,lagarias2000orthonormal,kolountzakis2000packing}.

Research on spectral sets
\cite{lagarias1997spectral,jorgensen1999spectral-pairs,laba2002spectral,laba2001twointervals,
bose2010spectrum,bose2010three,kolountzakis2011periodic,iosevich2011periodicity}
has been driven for many years by a conjecture of Fuglede
\cite{fuglede1974operators},
sometimes called the Spectral Set Conjecture,
which stated that a set $\Omega$ is spectral if and only
if it is a translational tile. A set $\Omega$ is a translational tile if
we can translate copies of $\Omega$ around and fill space without overlaps.
More precisely there exists a set $S \subseteq \RR^d$ such that
\beql{tiling}
\sum_{s\in S} \chi_\Omega(x-s) = 1,\ \ \mbox{a.e. ($x$)}.
\eeq

One can generalize naturally the notion of translational tiling from sets to functions
by saying that a nonnegative $f \in L^1(\RR^d)$ tiles when translated at the locations $S$
if $\sum_{s\in S} f(x-s) = \ell$ for almost every $x\in\RR^d$
(the constant $\ell$ is called the {\em level} of the tiling).
Thus the question of spectrality for a set $\Omega$ is essentially a tiling question
for the function $\Abs{\ft{\chi_\Omega}}^2$ (the {\em power-spectrum}).
Taking into account the equivalent condition \eqref{tiling-condition}
one can now, more elegantly, restate the Fuglede Conjecture as the equivalence
\beql{fuglede-conjecture}
\chi_\Omega \mbox{ tiles $\RR^d$ by translation at level 1} \Longleftrightarrow
\Abs{\ft{\chi_\Omega}}^2 \mbox{ tiles $\RR^d$ by translation at level $\Abs{\Omega}^2$}.
\eeq
In this form the conjectured equivalence is perhaps more justified.
However this conjecture is now known to be false in both directions if $d\ge 3$
\cite{tao2004fuglede,matolcsi2005fuglede4dim,kolountzakis2006hadamard,kolountzakis2006tiles,farkas2006onfuglede,farkas2006tiles},
but remains open in dimensions $1$ and $2$ and it is not out of the question
that the conjecture is true in all dimensions if one restricts the domain $\Omega$ to be convex. (It is known
that the direction ``tiling $\Rightarrow$ spectrality'' is true in the case of convex domains;
see for instance \cite{kolountzakis2004milano}.)
The equivalence \eqref{fuglede-conjecture} is also known, from the time of Fuglede's
paper \cite{fuglede1974operators}, to be true if one adds the word {\em lattice} to both sides
(that is, lattice tiles are the same as sets with a lattice spectrum).

\subsection{Orthogonal exponentials for the disk}
Already in \cite{fuglede1974operators} it was claimed that the disk in the plane (and the Euclidean ball in $\RR^d$)
is not a spectral set, in agreement with \eqref{fuglede-conjecture}.
A proof appeared in \cite{iosevich-katz-pedersen}.
Later it was proved in \cite{fuglede-ball,iosevich-rudnev-smooth-bodies}
that any orthogonal set of exponentials for the ball must necessarily be finite.
It is still unknown however if there is a uniform bound for the size of each
orthogonal set. It is still a possibility that there are arbitrarily large orthogonal sets
of exponentials for the ball and proving a uniform upper bound is probably very hard
as it appears to depend on algebraic relations among the roots of the Bessel function $J_1$.
In the direction of showing upper bounds for orthogonal sets of exponentials it
was proved in \cite{iosevich-jaming} that if $\Lambda$ is a set of orthogonal exponentials
for the ball then $\Abs{\Lambda \cap [-R,R]^d} = O(R)$, with the implicit
constant independent of $\Lambda$.
Completeness would of course require that $\Abs{\Lambda \cap [-R,R]^d} \gtrsim R^d$ (this follows easily
from the tiling condition \eqref{tiling-condition}).

The result in this paper, Theorem \ref{th:main} below, improves the result of  \cite{iosevich-jaming}
mentioned above. We choose to work only in the case of the unit disk in the plane and not in higher dimension
or in the larger class of smooth convex bodies in order to present a clear geometric argument, which
probably extends to these cases as well.

\begin{theorem}\label{th:main}
There are constants $C_1, C_2$ such that
whenever $\Lambda \subseteq \RR^2$ is an orthogonal set of exponentials for the unit disk in the plane and
$$
t = \inf\Set{\Abs{\lambda-\mu}: \lambda, \mu \in \Lambda, \lambda\neq\mu}
$$
then $\Abs{\Lambda} \le C_1 t$.
(It is well known and easy to see from \eqref{packing-condition} or \eqref{zeros-condition} that $t>0$.)

Furthermore, $\Abs{\Lambda \cap [-R,R]^2} \le C_2 R^{2/3}$ for all $R\ge 1$.
\end{theorem}

The proof of Theorem \ref{th:main} is given in the next section.

\section{Proof of the main theorem}

A crucial ingredient of the proof is the asymptotics for the zeros of the Fourier Transform
of the indicator function of the unit disk $D = \Set{x \in \RR^2: \Abs{x}\le 1}$,
which is of course a radial function.
Since the zeros of $\ft{\chi_D}(r)$ are the same as the zeros of the Bessel function $J_1(2\pi r)$
and since for the zeros of $J_1$, written as $j_{1,n}$, $n=1,2,\ldots$ we have an asymptotic expansion
\cite{abramowitz-stegun}
\beql{zero-asymptotics}
j_{1,n} = \rho_n + \frac{K_1}{\rho_n} + O\left(\frac{1}{\rho_n^3}\right),
	\ \ \mbox{where $\rho_n=n\pi + \frac{\pi}{4}$, $n=1,2,\ldots$},
\eeq
where $K_1$ is an absolute constant,
it follows that the zeros of $\ft{\chi_D}(r)$ are at the locations
\beql{ft-zeros}
r_n = \frac{1}{2\pi} j_{1,n} = \frac{n}{2}+\frac{1}{8} + \frac{K_1}{2\pi \rho_n} + O(n^{-3}).
\eeq
Moreover, if $0\le m-n \le K$ and $m,n \to \infty$ it follows from \eqref{zero-asymptotics} that
\beql{quadratic-error}
r_m - r_n = \frac{m-n}{2} + O\left(K n^{-2}\right) = \frac{m-n}{2} + O\left(Kr_n^{-2}\right) =
	\frac{m-n}{2} + O\left((r_m-r_n)r_n^{-2}\right).
\eeq

\begin{lemma}\label{lm:angle}
There is a constant $C>0$ such that whenever
$a, b, c \in \RR^2$ are orthogonal for the unit disk, with 
$\Abs{a-c}, \Abs{b-c}, \Abs{a-b} \ge R$ then the two largest
angles of the triangle $abc$ (as well as all its external angles) are
\beql{angle-bound}
\ge \frac{C}{R^{1/2}}.
\eeq
\end{lemma}

\begin{proof}
Assume without loss of generality that $R=\Abs{a-c}\le\Abs{b-c}\le\Abs{a-b}$
(see Fig.\ \ref{fig:triangle}).
Writing $\theta = \ft{bac}$ for the second largest angle and $T=\Abs{a-b}$ we have

\begin{figure}[h]
\begin{center}
\begin{asy}
import graph;
import markers;
size(8cm);
real off=0.3;
pair	a=(-1,0), b=(1,0), c=(-off,0.2);
pair	d=(-off,0);
dot(a); dot(b); dot(c);
draw(a -- b -- c -- a);
draw(c -- d, dashed);
draw("$T$", (a-(0,0.1)) -- (b-(0,0.1)), dashed, Arrows);
markangle("$\theta$", b, a, c);
label("$a$", a, W);
label("$b$", b, E);
label("$c$", c, N);
//label("$d$", d, NE);
label("$R$", 0.5*(a+c), NW);
\end{asy}

\mycaption{Three points orthogonal for the unit disk}{fig:triangle}
\end{center}
\end{figure}

$$
\Abs{b-c} = \sqrt{(T-R\cos\theta)^2 + R^2 \sin^2\theta} = \sqrt{(T-R)^2 + 2T R(1-\cos\theta)}
$$
from which we get
\beql{tmp-1}
\Abs{b-c}-(T-R) = \frac{2T R(1-\cos\theta)}{T-R+\Abs{b-c}} =
 \frac{2R(1-\cos\theta)}{1-\frac{R}{T} + \frac{\Abs{b-c}}{T}} \le 2R(1-\cos\theta) \le R\theta^2.
\eeq
From \eqref{ft-zeros} it follows that
as $R\to\infty$ the quantities $\Abs{a-b}, \Abs{b-c}, \Abs{a-c}$ are all of the form
$$
\frac{k}{2} + \frac{1}{8} + o(1),\ \ \ \mbox{for some integer $k$}.
$$
It follows that $\Abs{b-c}-(T-R) = \frac{k}{2} + \frac{1}{8} + o(1)$, for some integer $k\ge 0$.
This, together with \eqref{tmp-1}, implies that $\frac{k}{2} + \frac{1}{8} + o(1) \le R \theta^2$ which
gives us the required inequality with constant $C$ arbitrarily close to $\sqrt{1/8}$ when $R$ is large.
\end{proof}

\begin{corollary}\label{cor:strip}
There is a constant $C'>0$ such that whenever
$a, b, c \in \RR^2$ are orthogonal for the unit disk and their pairwise distances are at least $L$
then they cannot all belong to a strip of width $C' L^{1/2}$.
\end{corollary}

\begin{proof}
Suppose they do belong to such a strip.

\begin{figure}[h]
\begin{center}
\begin{asy}
import graph;
import markers;
size(8cm);
pair	A=(0,0), u=(10,5);
pair	a=A+0.1*u, b=A+0.9*u, c=A+(0,1)+0.3*u;
draw(A -- A+u);
draw(A+(0,1.5) -- A+(0,1.5)+u);
draw(a -- b -- c -- a);
dot(a); dot(b); dot(c);
label("$a$", a, SE);
label("$b$", b, SE);
label("$c$", c, NW);
markangle("$\theta$", b, a, c);
\end{asy}

\mycaption{Three points in a strip}{fig:strip}
\end{center}
\end{figure}

Move and turn the strip so that two of the points, those with the largest distance apart,  say $a$ and $b$
are on one of the strip sides and the other point $c$ is still in the strip (see Fig.\ \ref{fig:strip}).
Assume also that $c$ is closer to $a$ than to $b$.
By Lemma \ref{lm:angle} it follows that
the angle $\theta = \ft{bac}$ is at least
$$
\frac{C}{\Abs{a-c}^{1/2}}
$$
from which we obtain that the distance of $c$ to the line $ab$
is at least $C \sqrt{a-c} \ge C \sqrt{L}$,
a contradiction if the constant $C'$ in the Lemma is sufficiently small.
\end{proof}

\begin{corollary}\label{cor:smallest-gap}
Suppose $\Lambda \subseteq \RR^2$ is a set of orthogonal exponentials for the unit disk, $R>0$ and let
$$
\Delta = \inf\Set{\Abs{\lambda-\mu}:\ \lambda, \mu \in \Lambda\cap[-R,R]^2}.
$$
Then
\beql{bound-wrt-gap}
\Abs{\Lambda \cap [-R, R]^2} \le C \frac{R}{\Delta^{1/2}},
\eeq
for some constant $C>0$.
\end{corollary}

\begin{proof}
Cover $[-R,R]^2$ by $O(R/\Delta^{1/2})$ strips of width $c\Delta^{1/2}$, for small $c>0$.
From Corollary \ref{cor:strip} each of these contains at most two points of $\Lambda$.
\end{proof}

We may assume from now on
that the points $V=(\Delta, 0)$ and $-V=(-\Delta, 0)$ belong to the set $\Lambda$ and that
$t/2 \le \Delta \le t$.
It is also sufficient to bound the size of $\Lambda$ in the first quadrant only, for reasons
of symmetry, so we restrict ourselves to the first quadrant.
By Corollary \ref{cor:strip} we have that
\beql{two-strips}
\Abs{\Lambda \cap \Set{(x,y):\ x,y \ge 0,\ \min\Set{x,y} \le \Delta}} = O(\Delta^{1/2}).
\eeq
So from now on we may assume that the point $\lambda=(x,y) \in \Lambda$ belongs to the first quadrant and
has $x, y \ge \Delta$.

To each $\lambda = (x,y)$ in the open first quadrant we correspond two numbers $a(\lambda), b(\lambda) \in (0,\Delta)$
such that $a(\lambda)^2 + b(\lambda)^2 = \Delta^2$ and $\lambda$ is on the hyperbola
$$
H_\lambda:\ \frac{x^2}{a(\lambda)^2} - \frac{y^2}{b(\lambda)^2} = 1,\ \ \ (x,y\ge 0).
$$
This hyperbola $H_\lambda$ is the locus of all points $p$ in the first quadrant such that
\beql{locus-1}
\Abs{p+V} - \Abs{p-V} = 2a(\lambda).
\eeq
A parametrization of $H_\lambda$ is
\beql{param}
x(t) = a(\lambda) \cosh t,\ y(t) = b(\lambda) \sinh t,\ \ (t\ge 0).
\eeq
It follows that in the region of interest $x, y \ge \Delta$ we have
\beql{sizes}
\Delta \le x \le a(\lambda) e^t \le 2\Abs{\lambda},\ \ \ \Delta \le y \le b(\lambda) e^t \le 2\Abs{\lambda}.
\eeq

\begin{lemma}\label{lm:b-range}
There is a constant $K>0$ such that $b(\lambda) \ge K \Delta^{1/2}$ with the exception of at most a constant
number of points of $\Lambda$.
\end{lemma}

\begin{proof}
\ 

\begin{figure}[h]
\begin{center}
\begin{asy}
import graph;
import contour;
import markers;
size(9cm);

real    a=8.0, c=10.0;
real    b, factor=3;
b = sqrt(c^2-a^2);

real[]  v={1.0};

real f(real x, real y) {return x^2/a^2 - y^2/b^2;}

real wdt=factor*c;
real hgt=(b/a)*factor*c;
real t=2.0;
real x=a*cosh(t);
real y=b*sinh(t);
fill((0,0)--(wdt+c,0)--(wdt,hgt)--cycle, mediumgray);
dot("$\lambda$", (x,y));
draw(contour(f, (0, 0), (factor*c, factor*c), v));
draw("$L$", (0,0) -- (wdt, hgt), NW);
draw("$M$", (c,0) -- (wdt+c, hgt));
draw("$2\Delta$",(-c,-0.4*c) -- (c,-0.4*c), dashed, Arrows, PenMargins);
markangle("$c\Delta^{-1/2}$", (2*c,0), (c,0), (wdt+c, hgt));

xaxis("$x$");
dot("$-V$", (-c,0), SW);
dot("$V$", (c,0), SE);
label("$a(\lambda)$", (a,0), S);
label("0", (0,0), NW);
yaxis("$y$");
\end{asy}

\mycaption{Figure to aid the proof of Lemma \ref{lm:b-range}}{fig:hyperbola}
\end{center}
\end{figure}

It follows from Lemma \ref{lm:angle} that $\lambda$ cannot belong to the sector defined by the 
$x$-axis from $V$ onward and the straight line $M$ through $V$ of angle $c\Delta^{-1/2}$, if $c>0$ is
small enough (refer to Fig.\ \ref{fig:hyperbola}).
Now draw a parallel line $L$ to straight line $M$ through the origin and note that the strip
bordered by these two parallel lines, $L$ and $M$, has width $O(\Delta^{1/2})$.
Therefore, by Corollary \ref{cor:strip}, there is only a constant
number of elements of $\Lambda$ that can belong to the sector defined by the positve $x$-semiaxis and the
straight line $L$ (shaded region in Fig.\ \ref{fig:hyperbola}).

Suppose now that $\lambda\in\Lambda$ is such that $b(\lambda) \le K \Delta^{1/2}$, for
an appropriately small constant $K$, so that we also have $a(\lambda)\ge \Delta/2$.
It follows that the asymptote to the hyperbola $H_\lambda$, with equation $y=(b(\lambda)/a(\lambda))x$,
has slope
at most $\frac{K}{2} \Delta^{-1/2}$ which implies that $\lambda$, lying below that asymptote, is contained
in the (shaded) sector mentioned above.
\end{proof}

Writing $H(a,\Delta)$ for the hyperbola $x^2/a^2 - y^2/b^2 = 1$, with $a^2+b^2=\Delta^2$, we consider
the finite family of confocal hyperbolas
\beql{hyperbolas}
H_k = H\left(\frac{k}{4}, \Delta\right),\ \ k=0,1,2,\ldots,\Floor{4\Delta}.
\eeq
The hyperbola $H_k$ is the locus of all points $p$ with $\Abs{p+V} - \Abs{p-V} = \frac{k}{2}$.
For each $\lambda$ we define the corresponding $k$ to be the unique integer such that
\beql{def-k}
\Abs{\lambda+V} - \Abs{\lambda-V} = \frac{k}{2} + 2\epsilon,\ \ \ -\frac{1}{8} \le \epsilon < \frac{1}{8}.
\eeq
We write $a = k/4$, $b = \sqrt{\Delta^2 - a^2}$. It follows from \eqref{locus-1} and \eqref{def-k} that
$$
a(\lambda) = a+\epsilon,\ \ \ b(\lambda) = b-\epsilon',
$$
for some $\epsilon'$, of the same sign as $\epsilon$.
From \eqref{quadratic-error} we have that
\beql{error}
\Abs{\epsilon} \le C\Delta\Abs{\lambda}^{-2},
\eeq
for some absolute finite constant $C>0$.

Next we estimate $\epsilon'$:
\begin{align*}
\Abs{\epsilon'} & = \Abs{b-b(\lambda)}\\
 &= \Abs{\sqrt{\Delta^2-a^2}-\sqrt{\Delta^2-a(\lambda)^2}}\\
 &= \Abs{\frac{a(\lambda)^2-a^2}{b+b(\lambda)}}\\
 &= \Abs{\epsilon \frac{a+a(\lambda)}{b+b(\lambda)}}\\
 &\le \Abs{\epsilon} \frac{2\Delta}{b(\lambda)}\\
 &= O(\Abs{\epsilon} \Delta^{1/2}) & \mbox{(from Lemma \ref{lm:b-range}, excepting finitely many $\lambda$s)}\\
 &= O(\Delta^{3/2}\Abs{\lambda}^{-2}) & \mbox{(from \eqref{error})}.
\end{align*}

The asymptote $L(a, \Delta)$ to $H(a, \Delta) = H_k$ is the line $y = (b/a)x$
and a unit normal vector to this line is $u=(b/\Delta, -a/\Delta)$.
We can bound the distance of
$$
\lambda = (x,y) = (a(\lambda)\cosh t, b(\lambda) \sinh t)
$$
to $L(a, \Delta)$
as follows:
\begin{align*}
\Abs{u\cdot\lambda} &= \Abs{ \frac{bx}{\Delta} - \frac{ay}{\Delta} }\\
 &= \frac{1}{\Delta} \Abs{b a(\lambda) \cosh t - a b(\lambda) \sinh t}\\
 &= \frac{1}{\Delta} \Abs{(b(\lambda)+\epsilon') a(\lambda) \cosh t - (a(\lambda)-\epsilon) b(\lambda) \sinh t}\\
 &= \frac{1}{\Delta} \Abs{a(\lambda)b(\lambda) e^{-t} - \epsilon b(\lambda) \sinh t + \epsilon' a(\lambda) \cosh t}\\
 &= O(\Delta^2 \Abs{\lambda}^{-1}) + \Abs{\epsilon} O\left(\Delta^{-1} b(\lambda) \sinh t + \Delta^{-1/2} a(\lambda) \cosh t\right)\\
	& \ \ \ \ \ \ \mbox{(since $\epsilon'=O(\Delta^{1/2}\epsilon)$ and $a(\lambda)e^t \sim \Abs{\lambda}$ or $b(\lambda)e^t \sim \Abs{\lambda}$ from \eqref{sizes})}\\
 &= O(\Delta^2 \Abs{\lambda}^{-1}) + \Abs{\epsilon} O\left( \Delta^{-1} b(\lambda)e^t + \Delta^{-1/2} a(\lambda)e^t\right)\\
 &= O(\Delta^2 \Abs{\lambda}^{-1}) + O\left( \Abs{\epsilon} \Abs{\lambda}\Delta^{-1/2}\right)
        & \mbox{(from \eqref{sizes})}\\
 &= O\left(\Delta^2 \Abs{\lambda}^{-1}\right)
	& \mbox{(from \eqref{error})}.
\end{align*}
Therefore in the region $\Abs{\lambda} \ge C \Delta^{3/2}$ each point of $\Lambda$ is at distance
$O(\Delta^{1/2})$ from one of the asymptotes to the hyperbolas $H_k$. In each strip of width $O(\Delta^{1/2})$ around
each such asymptote we therefore have at most $C$ points, a constant. This gives a total of $O(\Delta)$ points
of $\Lambda$ in that region as there are that many hyperbolas $H_k$.
In the region $\Abs{\lambda} \le C \Delta^{3/2}$ we also have $O(\Delta)$ points because of Corollary
\ref{cor:smallest-gap}.
This concludes the proof of the first part of Theorem \ref{th:main}.

To prove that $\Abs{\Lambda \cap [-R,R]^2} = O(R^{2/3})$ notice that by Corollary \ref{cor:smallest-gap}
and by the first part of Theorem \ref{th:main} we have
$$
\Abs{\Lambda \cap [-R,R]^2} = O\left( \min\Set{\frac{R}{t^{1/2}}, t} \right) = O\left(R^{2/3}\right).
$$

\bibliographystyle{abbrv}
\bibliography{spectral-sets}

\end{document}